\documentclass[11pt]{article}
\usepackage[margin=1in]{geometry}
\usepackage[utf8]{inputenc}
\usepackage{amsmath,amsfonts,amssymb,amsthm}
\usepackage{mathrsfs}
\usepackage[dvipsnames]{xcolor}
\usepackage{mathtools}
\usepackage{graphicx}
\usepackage{tikz-cd}
\usepackage[hidelinks]{hyperref}
\usepackage[shortlabels]{enumitem}
\usepackage{bbm}
\usepackage{float}
\usepackage{slashed}

\setlist{nolistsep}

\newcommand{\drm}{\mathrm{d}}
\newcommand{\Cl}{\mathbb{C}\ell}
\newcommand{\hatotimes}{\mathbin{\hat\otimes}}

\DeclareMathOperator{\End}{End}
\DeclareMathOperator{\Span}{Span}
\DeclareMathOperator{\ind}{ind}

\DeclareMathOperator{\tr}{tr}

\newtheorem{thm}{Theorem}[section]
\newtheorem{lem}[thm]{Lemma}
\newtheorem{prp}[thm]{Proposition}
\newtheorem{cor}{Corollary}[thm]
\newtheorem*{clm}{Claim}

\theoremstyle{definition}

\theoremstyle{remark}

\title{Scalar curvature rigidity of parabolically convex domains in hyperbolic spaces}
\author{Chengzhang Sun}
\date{}

\begin{document}
    \maketitle
    \begin{abstract}
        For a parabolically convex domain $M\subseteq \mathbb{H}^n$, $n\ge 3$, we prove that if $f:(N,\bar g)\to (M,g)$ has nonzero degree, where $N$ is spin with scalar curvature $R_N\ge -n(n-1)$, and if $f|_{\partial N}$ does not increase the distance and the mean curvature, then $N$ is hyperbolic, and $\partial N$ is isometric to $\partial M$. This is a partial generalization of Lott's result \cite{lott2021index} to negative lower bounds of scalar curvature. 
    \end{abstract}
    \section{Introduction}

    The rigidity of scalar curvature has been extensively studied. The first result in this direction was Llarull's theorem, which was generalized by Goette and Semmelmann \cite{goette2002scalar} to even dimensional manifolds with nonnegative curvature operator. Later it was generalized to even-dimensional manifolds with smooth boundaries in \cite{lott2021index}, and manifolds with polyhedral boundaries in \cite{wang2021gromov,brendle2024scalar}. By a result of Lohkamp \cite{lohkamp1999scalar}, the positive mass theorem can be deduced by scalar curvature rigidity, see \cite{wang2021gromov} for example. The aforementioned papers have assumed necessary spin conditions. There are also proofs of such rigidity results using minimal surface technics, for example, \cite{li2024dihedral}. 

    For negative scalar curvature, there are some recent results. B\"ar, Brendle, Hanke and Wang \cite{bar2024scalar} proved the scalar curvature rigidity for warped product metrics, which applies to cylinders in hyperbolic spaces. Wang and Xie \cite{wang2023scalar} proved the scalar curvature rigidity for degenerate warped product metrics, which applies to balls in hyperbolic spaces, and they also proved the scalar curvature rigidity for radially convex domains in hyperbolic spaces \cite{wang2023dihedral}. 

    In this paper, we prove a scalar curvature rigidity result for parabolically convex domains. Parabolic convexity is a natural condition proposed by Gromov \cite{gromov2018scalar}. We are using the definition from \cite{chai2023scalar}. Let $\mathbb{R}^n_+=\{(x^1,\ldots,x^n)\in\mathbb{R}^n:x^1>0\}$ be the open upper half space. Let $g_0$ and $g_1$ be the Euclidean metric and the hyperbolic metric on $\mathbb{R}^n_+$ respectively, i.e., 
    $$g_1=\frac{1}{(x^1)^2}g_0=\frac{(\drm x^1)^2+\cdots+(\drm x^n)^2}{(x^1)^2}.$$
    A domain $M\subseteq\mathbb{R}^n_+$ is \emph{parabolically convex} if it is equipped with the hyperbolic metric $g_1$, and $M$ is convex with respect to the Euclidean metric $g_0$. It should be noted that this definition depends on the upper half space model. 

    There are other notions of convexity in hyperbolic geometry that are more intrinsic. A strong convexity is called horo-convexity, see \cite{borisenko1999total} for example. A domain $D$ in a hyperbolic space is horo-convex if every point on the boundary has a supporting horosphere. Since horoballs are parabolically convex, a horo-convex domain is also parabolically convex. 

    In this paper, the domain $M$ is assumed to have a smooth boundary. It is natural to ask if the result holds if $M$ has a polyhedral boundary, and we will answer it in a separate paper. Some results in this direction are \cite{li2020dihedral,chai2023scalar}, with certain dimension restrictions. In \cite{chai2024scalar}, they proved the dihedral angle rigidity in all dimensions under the matching angle condition. 

    The paper is organized as follows. In \S2 we state the main theorem. Some basic calculations about the modified connection are given in \S3. In \S4 we deduce the main theorem assuming the existence of a nontrivial harmonic spinor. In \S5 we prove the existence of such spinors by index calculations. 

    \paragraph{Acknowledgement.} The author expresses gratitude to Xianzhe Dai and Yihan Li for their insightful discussions and to Xiaoxiang Chai for bringing the recent paper \cite{chai2024scalar} to their attention.

    \section{Setup and the main theorem}
    
    Let $\mathbb{H}^n=(\mathbb{R}^n_+,g_1)$ be the hyperbolic space. Assume that $M$ is a compact domain in $\mathbb{H}^n$ with smooth boundary $\partial M$. Let $f:N\to M$ be a smooth map, where $(N,\bar g)$ is an $n$-dimensional compact, connected, spin Riemannian manifold with smooth boundary $\partial N$. Assume that $f(\partial N)\subseteq\partial M$. Let $\partial f:\partial N\to\partial M$ be the restriction of $f$. 

    \begin{thm}\label{thm:main}
        Assume that $n\ge 3$ and 
        \begin{enumerate}
            \item The hypersurface $\partial M$ is convex in $(\mathbb{R}^n_+,g_0)$, i.e., $M$ is parabolically convex. 
            \item The map $\partial f$ is $1$-Lipschitz with respect to $g_1$.
            \item The map $\partial f$ does not increase the mean curvature, i.e., $H_{\partial N}\ge H_{\partial M,g_1}\circ \partial f$. 
            \item The map $f$ has nonzero degree. 
            \item The map $f$ does not increase the scalar curvature, i.e., $R_{N}\ge R_{M,g_1}\circ f=-n(n-1)$. 
        \end{enumerate}
        Then 
        \begin{enumerate}
            \item The manifold $N$ has constant sectional curvature $-1$. 
            \item The hypersurface $\partial N$ has mean curvature $H_{\partial N}= H_{\partial M,g_1}\circ \partial f$.
            \item The map $\partial f$ is an isometry (when restricted to each component of $\partial N$). 
        \end{enumerate}
    \end{thm}

    \noindent Note that we only assume that $\partial f$ is $1$-Lipschitz. As a result, we cannot conclude the rigidity of $f$ in the interior. The result can also be viewed as a mapping version of \cite[Theorem 1.2]{chai2024scalar}. We will prove the theorem in two cases according to the parity of $n$. In section 3, we will introduce the modified connection and the Lichnerowicz formula, and establish some basic estimates. In section 4, we will deduce the main theorem assuming the existence of a nontrivial harmonic spinor. The proof of the existence of such spinors will be given in section 5, as a result of index-theoretic calculations.

    \section{The modified connection}

    \subsection{The twisted spinor bundle}

    Let $\slashed{S}_N$ be the complex spinor bundle on $N$. For $g=g_0$ or $g_1$, denote the complex spinor bundle on $M$ as $\slashed{S}_{M,g}$. Set $S_g=\slashed{S}_N\otimes f^*\slashed{S}_{M,g}$. The subscript $g$ will be suppressed if the discussion applies to both $g_0$ and $g_1$. 

    There is an induced connection on $S$ given by 
    $$\nabla=\nabla^{\slashed{S}_N}\otimes 1+1\otimes \nabla^{f^*\slashed{S}_M},$$
    where the $1$'s stand for identity maps. In other words, 
    $$\nabla(\alpha\otimes\beta)=\nabla^{\slashed{S}_N}\alpha\otimes\beta+\alpha\otimes\nabla^{f^*\slashed{S}_M}\beta, \quad \alpha\in C^\infty(N,\slashed{S}_N), \quad \beta\in C^\infty(N,f^*\slashed{S}_M).$$

    Let $\Cl_N$ and $\Cl_M$ be the complex Clifford bundles. Denote the Clifford multiplication by $c$. If $n$ is odd, we view $S$ as a module over $\Cl_N\otimes f^*\Cl_M$, i.e., 
    \[c(\varphi\otimes\psi)(\alpha\otimes\beta)=c(\varphi)\alpha\otimes c(\psi)\beta.\]
    The (twisted) Dirac operator on $S$ is defined as 
    $$D=\sum_{\alpha=1}^n c(\bar e_\alpha\otimes 1)\nabla_{\bar e_\alpha},$$
    where $\{\bar e_\alpha:1\le \alpha\le n\}$ is a local orthonormal frame of $TN$. We will write 
    $$\bar c(\bar v)=c(\bar v\otimes 1), \quad c(v)=c(1\otimes v),$$
    where $\bar v\in TN$ and $v\in TM$ or $f^*TM$. Note that 
    $$\bar c(\bar v)c(v)=c(v)\bar c(\bar v), \quad \bar c(\bar v)^2=-|\bar v|^2, \quad \bar c(\bar v)^*=-\bar c(\bar v), \quad c(v)^2=-|v|^2, \quad c(v)^*=-c(v).$$

    If $n$ is even, we consider the graded vector bundle 
    $$S=\slashed{S}_{N}\hatotimes f^*\slashed{S}_{M},$$
    over $N$, where we take into account the even-odd grading on $\slashed{S}_N$ and $\slashed{S}_M$. It is the same vector bundle as in the ungraded case. Now $S$ as a graded module over $\Cl_N\hatotimes f^*\Cl_M=\Cl(TN\oplus f^*TM)$, i.e.,
    \[c(\varphi\otimes\psi)(\alpha\otimes\beta)=(-1)^{\deg \psi\deg\alpha}c(\varphi)\alpha\otimes c(\psi)\beta.\]
    Note that we write $\otimes$ between elements or sections, and write $\hatotimes$ between algebras or bundles. We will use the convention that the metric on $TN\oplus f^*TM$ is $g_N\oplus(-f^*g_M)$, c.f., \cite{cecchini2024rigidity}. The connection on $S$ is the same as in the ungraded case, i.e.,
    $$\nabla(\alpha\otimes\beta)=\nabla^{\slashed{S}_N}\alpha\otimes\beta+\alpha\otimes\nabla^{f^*\slashed{S}_M}\beta, \quad \alpha\in C^\infty(N,\slashed{S}_N), \quad \beta\in C^\infty(N,f^*\slashed{S}_M).$$
    The (twisted) Dirac operator on $S$ is defined as 
    $$D=\sum_{\alpha=1}^{n} c(\bar e_\alpha \otimes 1)\nabla_{\bar e_\alpha}.$$
    In this setting, we will write 
    $$\bar c(\bar v)=c(\bar v\otimes 1), \quad c(v)=c(1\otimes v),$$
    where $\bar v\in TN$ and $v\in TM$ or $f^*TM$. Note that 
    $$\bar c(\bar v)c(v)=-c(v)\bar c(\bar v), \quad \bar c(\bar v)^2=-|\bar v|^2, \quad \bar c(\bar v)^*=-\bar c(\bar v), \quad c(v)^2=|v|^2, \quad c(v)^*=c(v).$$

    In many articles such as \cite{goette2002scalar,lott2021index,wang2021gromov}, it was not specified whether the tensor product is graded or not. The recent work \cite{cecchini2024rigidity} has emphasized that the tensor product should be graded. We find that in the even dimensional case, the graded tensor product is more suitable. 

    \subsection{The Lichnerowicz formula}

    In both cases, $(S,D)$ is a Dirac bundle over $N$ in the sense of \cite{lawson2016spin}. The Lichnerowicz formula for $D$ on the bundle $S$ is given by 
    \begin{equation}\label{eq:Lich-1}
        D^2=\nabla^*\nabla+\frac{1}{4}R_N+c(R^{f^*\slashed S_M}).
    \end{equation}
    Let $\{\bar\omega_{\rho}:1\le\rho\le\binom{n}{2}\}$ and $\{\omega_r:1\le r\le\binom{n}{2}\}$ be local orthonormal frames of $\Lambda^2TN$ and $\Lambda^2TM$. Then in both cases,  
    \begin{equation}\label{eq:twist-curv}
        c(R^{f^*\slashed S_M})=-\frac{1}{2}\sum_{\rho,r}g(\mathcal{R}^M(f_*\bar\omega_\rho),\omega_r)c(\bar\omega_\rho)\otimes c(\omega_r),
    \end{equation}
    c.f., \cite[Equation (1.5)]{goette2002scalar}. Note that the Clifford multiplication by 2-forms is defined by
    $$c(u\wedge v)=c(u)c(v)$$
    for any $u\perp v$. 

    \subsection{The modified connection}

    The idea of using modified connections on spinor bundles dates back to \cite{min1989scalar,andersson1998scalar}. Let $\{\bar e_\alpha:1\le \alpha\le n\}$ be a local orthonormal frame of $TN$. 

    \begin{prp}\label{prp:Lich}
        In both cases, let $V\in C^\infty(N,f^*TM)$. For $X\in C^\infty(N,TN)$, define 
        $$\widehat\nabla_X=\nabla_X+c(X\otimes V):C^\infty(N,S)\to C^\infty(N,S),$$
        and 
        $$\widehat{D}=\sum_{\alpha=1}^n \bar c(\bar e_\alpha)\widehat\nabla_{\bar e_\alpha}.$$
        Then 
        \begin{equation}\label{eq:Lich-2}
            \widehat{D}^*\widehat{D}=\widehat{\nabla}^*\widehat{\nabla}+\frac{1}{4}R_N+c(R^{f^*\slashed S_M})+n(n-1)|V|^2-(n-1)c(\nabla^{f^*TM}V),
        \end{equation}
        where $|V|^2=g(V,V)$ and 
        $$c(\nabla^{f^*TM}V)=\sum_{\alpha=1}^n \bar c(\bar e_\alpha)c(\nabla_{\bar e_\alpha}^{f^*TM}V).$$
    \end{prp}

    \begin{proof}
        We may assume that $\nabla\bar{e}_\alpha=0$ at some point $y\in N$. By a direct computation, at $y$, 
        \begin{align*}
            \widehat\nabla_{\bar e_\alpha}^*&=-\nabla_{\bar e_\alpha}+c(\bar e_\alpha\otimes V),\\
            \widehat\nabla_{\bar e_\alpha}^*\widehat\nabla_{\bar e_\alpha}&=\nabla_{\bar e_\alpha}^*\nabla_{\bar e_\alpha}-\bar c(\bar e_\alpha)c(\nabla_{\bar e_\alpha}^{f^*TM}V)+|V|^2,\\
            \widehat\nabla^*\widehat\nabla&=\sum_{\alpha=1}^n\widehat\nabla_{\bar e_\alpha}^*\widehat\nabla_{\bar e_\alpha}=\nabla^*\nabla-c(\nabla^{f^*TM}V)+n|V|^2.
        \end{align*}
        Similarly, 
        \begin{align*}
            \widehat D&=D-c(n\otimes V),\\
            \widehat D^*\widehat D&=D^2-n c(\nabla^{f^*TM}V)+n^2|V|^2.
        \end{align*}
        Now \eqref{eq:Lich-2} follows from \eqref{eq:Lich-1}. 
    \end{proof}

    \begin{cor}\label{cor:Lich}
        If $g=g_0$, $V=\pm\frac{1}{2}\frac{\partial}{\partial x^1}$, then 
        $$\widehat{D}^*\widehat{D}=\widehat{\nabla}^*\widehat{\nabla}+\frac{1}{4}\left(R_N+n(n-1)\right).$$
    \end{cor}

    \subsection{The mean curvature and the second fundamental form}

    Denote $\nu_N$ to be the unit inner normal vector field on $\partial N$. Similarly, $\nu_{M,g}$ is the unit inner normal vector field on $\partial M$ with respect to the metric $g=g_0$ or $g_1$. Define 
    $$\nu_f=\nu_{M,g}\circ \partial f\in C^\infty(\partial N,f^*TM).$$
    In the both cases, define
    $$\chi=c(\nu_N\otimes \nu_f):C^\infty(\partial N,S)\to C^\infty(\partial N,S).$$
    The mean curvatures of $\partial N$ and $\partial M$ are computed relative to the inner normal. For example, 
    $$H_{\partial N}=-\sum_{\alpha=1}^{n-1}\left<\nabla_{\bar e_\alpha}\nu_N,\bar e_\alpha\right>,$$
    where $\{\bar e_\alpha:1\le\alpha\le n-1\}$ is a local orthonormal frame of $T\partial N$.

    \begin{lem}\label{lem:mean}
        The mean curvature of $\partial M$ satisfies 
        $$H_{\partial M,g_1}=x^1H_{\partial M,g_0}+(n-1)\,\drm x^1(\nu_{M,g_0}).$$
    \end{lem}
    \begin{proof}
        It is known that after a conformal change $g_1=e^{2u}g_0$, the mean curvature becomes
        $$H_{\partial M,g_1}=e^{-u}(H_{\partial M,g_0}-(n-1)g_0(\nabla u,\nu_{M,g_0})).$$
        Apply this to $u=-\ln x^1$ and we obtain the desired formula. 
    \end{proof}

    Define a Dirac operator on the boundary by
    $$D^{\partial N}=\sum_{\alpha=1}^{n-1}\bar c(\nu_N)\bar c(\bar e_\alpha)\nabla_{\bar e_\alpha}-\frac{1}{2}H_{\partial N}:C^\infty(\partial N,S)\to C^\infty(\partial N,S),$$
    Motivated by \cite{brendle2024scalar}, define 
    $$\mathcal B=\chi D^{\partial N}+D^{\partial N}\chi:C^\infty(\partial N,S)\to C^\infty(\partial N,S),$$
    which plays the role of the second fundamental form. 

    \begin{lem}\label{lem:B}
        In both cases, the operator $\mathcal B$ can be written as 
        $$\mathcal B=\sum_{\alpha=1}^{n-1}\bar c(\bar e_\alpha)c(\nabla_{\bar e_\alpha}^{f^*TM}\nu_f).$$
    \end{lem}
    \begin{proof}
        By a direct calculation, 
        $$\mathcal B=\sum_{\alpha=1}^{n-1}\left(\bar c(\bar e_\alpha)c(\nabla_{\bar e_\alpha}^{f^*TM}\nu_f)+\bar c(\bar e_\alpha)\bar c(\nabla_{\bar e_\alpha}^{TN}\nu_N)\chi\right)-H_{\partial N}\chi.$$
        Note that 
        \begin{align*}
            \sum_{\alpha=1}^{n-1}\bar c(\bar e_\alpha)\bar c(\nabla_{\bar e_\alpha}^{TN}\nu_N)&=\sum_{\alpha,\beta=1}^{n-1}\left<\nabla_{\bar e_\alpha}^{TN}\nu_N,\bar e_\beta\right>\bar c(\bar e_\alpha)\bar c(\bar e_\beta)\\
            &=-\sum_{\alpha=1}^{n-1}\left<\nabla_{\bar e_\alpha}^{TN}\nu_N,\bar e_\alpha\right>+\sum_{\alpha\ne\beta}\left<\nabla_{\bar e_\alpha}^{TN}\nu_N,\bar e_\beta\right>\bar c(\bar e_\alpha)\bar c(\bar e_\beta)\\
            &=H_{\partial N}.
        \end{align*}
        Here, the second summation is zero because if one swaps $\alpha,\beta$, the term becomes its negative. Hence, the desired result follows. 
    \end{proof}

    \subsection{Some linear algebra}

    Let $T:\mathbb{R}^n\to\mathbb{R}^m$ be a linear map. Then there is an orthonormal basis $\{v_j:1\le j\le n\}$ of $\mathbb{R}^n$ and an orthonormal basis $\{w_i:1\le i\le m\}$ of $\mathbb{R}^m$ such that 
    \[Tv_i=\sigma_i w_i, \quad \sigma_1\ge\cdots\ge\sigma_r>\sigma_{r+1}=\cdots=0.\]
    The numbers $\sigma_i=\sigma_i(T)$ are the singular values of $T$, and $r$ is the rank of $T$. It is well known that 
    \begin{equation}\label{eq:max-min}
        \sigma_i=\max\{\min\{|Tu|:u\in U,|u|=1\}:U\subseteq\mathbb{R}^n,\dim U=i\},
    \end{equation}
    where $|\cdot|$ is the Euclidean norm. The \emph{Schattern $p$-norm} of $T$ is defined as
    $$\|T\|_p=\left(\sum_{i=1}^r\sigma_i^p\right)^{1/p}, \quad 1\le p<\infty. $$
    We will only use the $1$-norm and $2$-norm, which are also known as the trace norm and the Frobenius norm respectively. For $p=\infty$, $\|T\|_\infty=\sigma_1(T)$ is the operator norm of $T$. The following inequality is well known, but we also need the rigidity statement.
    \begin{lem}
        If $T:\mathbb{R}^n\to\mathbb{R}^m$ and $S:\mathbb{R}^m\to\mathbb{R}^k$ are linear maps, then 
        $$\|ST\|_p\le\sigma_1(S)\|T\|_p.$$
        Assume that $T$ is surjective and $p<\infty$. Then the equality holds if and only if $\sigma_1(S)=\sigma_m(S)$.
    \end{lem}
    \begin{proof}
        The inequality follows directly from \eqref{eq:max-min}, since for all $1\le i\le r$, 
        $$\sigma_i(ST)\le \sigma_1(S)\sigma_i(T).$$ 
        Now suppose equality holds and $T$ is surjective. Then $\sigma_i(T)=\sigma_i(ST)$ and we denote this common value as $\sigma_i$. We need to show that for all $v\in\mathbb{R}^m$, $|Sv|=\sigma_1(S)|v|$. 
        
        If $S=0$, the statement is clear. Hence, by rescaling, we may assume $\sigma_1(S)=1$. Let
        $$W=\{w\in\mathbb{R}^m:|Sw|=|w|\}.$$
        It is easy to see that $W$ is the subspace of $\mathbb{R}^m$ spanned by those singular vectors of $S$ corresponding to singular value $1$. It suffices to show that the range of $T$ is contained in $W$. 

        Suppose that $v_1\in\mathbb{R}^n$ is a unit vector with $|STv_1|=\sigma_1$. We have $Tv_1\in W$. For if $Tv_1\notin W$, then $|STu_1|<|Tu_1|\le\sigma_1$, which is a contradiction. 

        Suppose that we have shown that, for $1\le i<i_0$, if $v_i\in\mathbb{R}^n$ is a unit vector with $|STv_i|=\sigma_i$, then $Tv_i\in W$. We aim to prove that if $v_{i_0}\in\mathbb{R}^n$ is a unit vector with $|STv_{i_0}|=\sigma_{i_0}$, then $Tv_{i_0}\in W$. 

        If $\sigma_{i_0}=\sigma_{i_0-1}$, then there is nothing to prove. Thus, we assume that $\sigma_{i_0}<\sigma_{i_0-1}$. In particular, we can find an orthonormal system $\{v_i:1\le i<i_0\}\subset\mathbb{R}^n$ such that $|STv_i|=\sigma_i$. By \eqref{eq:max-min}, 
        $$\sigma_{i_0}\ge\min\{|Tv|:v\in\Span\{v_1,\ldots,v_{i_0}\},|v|=1\}.$$
        Take a minimizer $v$. If $Tv\notin W$, then $|STv|<|Tv|\le\sigma_{i_0}$ is a contradiction. Hence, $Tv\in W$. Since $Tv_1,\ldots,Tv_{i_0-1}\in W$ and $v\notin\Span\{v_1,\ldots,v_{i_0-1}\}$, we conclude that $Tv_{i_0}\in W$. 
    \end{proof}

    By taking transpose, we immediately get
    \begin{lem}\label{lem:norm}
        If $T:\mathbb{R}^n\to\mathbb{R}^m$ and $S:\mathbb{R}^m\to\mathbb{R}^k$ are linear maps, then 
        $$\|ST\|_p\le\sigma_1(T)\|S\|_p.$$
        Assume that $S$ is injective and $p<\infty$. Then the equality holds if and only if $\sigma_1(T)=\sigma_m(T)$.
    \end{lem}

    The definitions and results above can be easily generalized to any finite dimensional inner product spaces. For example, if $f:(N,\bar g)\to (M,g)$ is a smooth map between Riemannian manifolds, then
    $$\|\drm f\|_p:N\to\mathbb{R}$$
    is well defined. In our setting, there are two metrics $g_1=\frac{1}{(x^1)^2}g_0$ on $M$. It is easy to see that 
    \begin{equation}
        \|\drm f\|_{p,g_1}=\frac{1}{x^1}\|\drm f\|_{p,g_0}. 
    \end{equation}

    \subsection{Estimate of the second fundamental form}

    Let $g=g_0$ or $g_1$. Recall that the shape operator of $\partial M$ with respect to $g$ is 
    $$\mathcal S_g:T_p\partial M\to T_p\partial M,\quad X\mapsto -\nabla_X^{TM,g}\nu_{M,g},$$
    and the mean curvature is 
    $$H_{\partial M,g}=\tr \mathcal S_g.$$
    We assume that $M$ is convex in $(\mathbb{R}^n_+,g_0)$, i.e., $\|\mathcal S_{g_0}\|_1=\tr \mathcal S_{g_0}=H_{\partial M,g_0}$. 

    \begin{prp}\label{prp:B-estimate}
        For $y\in\partial N$, define 
        $$\drm\hat\nu_f(y):T_y\partial N\to T_{f(y)}\partial M, \quad X\mapsto \mathcal S_g(f_*X).$$ 
        Then for any $s\in C^\infty(\partial N,S)$, 
        $$|\left<\mathcal{B}s,s\right>|\le \|\drm\hat\nu_f\|_1\cdot|s|^2.$$
    \end{prp}
    \begin{proof}
        Fix any $y\in\partial N$. Take an orthonormal basis $\{v_1,\ldots,v_{n-1}\}$ of $T_y\partial N$ and an orthonormal basis $\{w_1,\ldots,w_{m-1}\}$ of $T_{f(y)}\partial M$, such that 
        $$\drm\hat\nu_f(v_i)=\sigma_i w_i,\quad \sigma_1\ge\cdots\ge\sigma_r>\sigma_{r+1}=\cdots=0.$$
        By Lemma \ref{lem:B}, 
        $$\mathcal{B}=\sum_{i=1}^{n-1}\bar c(v_i)c(\nabla_{v_i}^{f^*TM}\nu_f).$$
        Since $M$ is equipped with the Euclidean metric, we have, at the point $y$, 
        $$\nabla_{v_i}^{f^*TM}\nu_f=\nabla^{TM}_{f_*v_i}\nu_{M,g_0}=-\drm\hat\nu_f(v_i)=-\sigma_iw_i.$$
        Thus, at the point $y$, 
        $$\mathcal B=-\sum_{i=1}^{r}\sigma_i \bar c(v_i)c(w_i).$$
        Since $(\bar c(v_i) c(w_i))^2=1$, we have 
        $$|\left<\bar c(v_i) c(w_i)s,s\right>|\le |s|^2,$$
        for any $s\in S_y$. Hence, 
        \begin{align*}
            |\left<\mathcal Bs,s\right>|\le \sum_{i=1}^r\sigma_i|s|^2&=\|\drm\hat\nu_f\|_1\cdot|s|^2.\qedhere
        \end{align*}
    \end{proof}

    Next we estimate the norm of $\drm\hat\nu_f$. 

    \begin{lem}\label{lem:norm-estimate}
        Under the assumptions of the main theorem \ref{thm:main},
        $$\|\drm\hat\nu_{f,g_0}\|_{1,g_0}\le x^1 H_{\partial M,g_0}\circ \partial f.$$
        The equality holds at some $y\in\partial N$ if and only if $\partial f$ is a Riemannian submersion at $y$. 
    \end{lem}

    \begin{proof}
        Since $\drm\hat\nu_{f,g_0}=\mathcal S_{g_0}\circ(\partial f)_*$, and since $\partial f$ is $1$-Lipschitz with respect to $g_1$, we have 
        $$\|\drm\hat\nu_{f,g_0}(y)\|_{1,g_0}=x^1\|\drm\hat\nu_{f,g_0}(y)\|_{1,g_1}\le x^1\|\mathcal S_{g_0}(f(y))\|_{1,g_1},$$
        by Lemma \ref{lem:norm}. On the right hand side, $\mathcal S_{g_0}(f(y))$ is viewed as an endomorphism on $(T_{f(y)}\partial M,g_1)$, and it is easy to see that the norm is the same if we compute on the space $(T_{f(y)}\partial M,g_0)$. Hence,
        $$\|\drm\hat\nu_{f,g_0}(y)\|_{1,g_0}\le x^1\|\mathcal S_{g_0}(f(y))\|_{1,g_0}=x^1H_{\partial M,g_0}(f(y)).$$
        The rigidity statement follows from Lemma \ref{lem:norm}.
    \end{proof}

    Combining the lemmas, we conclude that, under the assumptions of the main theorem \ref{thm:main}, 
    $$|\left<\mathcal Bs,s\right>|\le (x^1H_{\partial M,g_0}\circ\partial f)|s|^2.$$

    \section{The boundary value problem}

    Throughout this section, $g=g_0$. Let $\epsilon=\pm 1$ be fixed. Set $$V=\frac{\epsilon}{2}\frac{\partial}{\partial x^1}$$ which is viewed as a constant vector field, and define modified connection as in Proposition \ref{prp:Lich}. We assume that
    \begin{enumerate}
        \item The hypersurface $\partial M$ is convex in $(\mathbb{R}^n_+,g_0)$, i.e., $M$ is parabolically convex. 
        \item The map $\partial f$ is $1$-Lipschitz with respect to $g_1$.
        \item The map $\partial f$ does not increase the mean curvature, i.e., $H_{\partial N}\ge H_{\partial M,g_1}\circ \partial f$. 
        \item The map $f$ has nonzero degree. 
        \item The map $f$ does not increase the scalar curvature, i.e., $R_{N}\ge R_{M,g_1}\circ f=-n(n-1)$. 
    \end{enumerate}
    as in the main theorem \ref{thm:main}.

    \subsection{Harmonic spinors are parallel}

    \begin{prp}\label{prp:parallel}
        Suppose that there is a nontrivial solution to the boundary value problem
        $$\begin{cases}
            \widehat Ds=0 & \text{ on }N,\\
            \chi s=\epsilon s & \text{ on }\partial N.
        \end{cases}$$
        Then the solution is parallel, i.e., $\widehat\nabla s=0$. Moreover, $R_N=-n(n-1)$, $H_{\partial N}=H_{\partial M,g_1}\circ \partial f$, and $\partial f$ is an isometry (when restricted to each component of $\partial N$). 
    \end{prp}

    \begin{proof}
        By Corollary \ref{cor:Lich}, using integration by parts, we get 
        \begin{equation}\label{eq:parallel-a}
            \int_N |\widehat Ds|^2-|\widehat \nabla s|^2-\frac{1}{4}\left(R_N+n(n-1)\right)|s|^2 =\int_{\partial N}\left<\widehat\nabla_{\nu_N}s+\bar c(\nu_N)\widehat Ds,s\right>. \tag{a}
        \end{equation}
        Recall that, for a local orthonormal frame $\{\bar e_\alpha:1\le\alpha\le n-1\}$ of $T\partial N$, 
        $$D=\bar c(\nu_N)\nabla_{\nu_N}+\sum_{\alpha=1}^{n-1}\bar c(\bar e_\alpha)\nabla_{\bar e_\alpha}, \quad D^{\partial N}=\sum_{\alpha=1}^{n-1}\bar c(\nu_N)\bar c(\bar e_\alpha)\nabla_{\bar e_\alpha}-\frac{1}{2}H_{\partial N}.$$
        Hence, 
        $$\nabla_{\nu_N}+\bar c(\nu_N)D=D^{\partial N}+\frac{1}{2}H_{\partial N}.$$
        For the modified connection, c.f. Proposition \ref{prp:Lich} and its proof, we have 
        \begin{equation}\label{eq:parallel-b}
            \widehat\nabla_{\nu_N}+\bar c(\nu_N)\widehat D=D^{\partial N}+\frac{1}{2}H_{\partial N}-(n-1)c(\nu_N\otimes V). \tag{b}
        \end{equation}
        Thus, we need to compute $\left<D^{\partial N}s,s\right>$ and $\left<c(\nu_N\otimes V)s,s\right>$. 

        Recall that $\mathcal B=\chi D^{\partial N}+D^{\partial N}\chi$. Note that $\chi s=\epsilon s$ and that $\chi$ is self-adjoint. We have 
        \begin{align*}
            \left<D^{\partial N}s,s\right>&=\epsilon\left<D^{\partial N}\chi s,s\right>=\epsilon\left<\mathcal Bs,s\right>-\epsilon\left<\chi D^{\partial N}s,s\right>\\
            &=\epsilon\left<\mathcal Bs,s\right>-\epsilon\left<D^{\partial N}s,\chi s\right>=\epsilon\left<\mathcal Bs,s\right>-\left<D^{\partial N}s,s\right>.
        \end{align*}
        Hence, 
        \begin{equation}\label{eq:parallel-c}
            \left<D^{\partial N}s,s\right>=\frac{1}{2}\epsilon\left<\mathcal Bs,s\right>.\tag{c}
        \end{equation}
        Similarly, set $\mathcal C=\chi c(\nu_N\otimes V)+c(\nu_N\otimes V)\chi$. Then 
        $$\left<c(\nu_N\otimes V)s,s\right>=\frac{1}{2}\epsilon\left<\mathcal Cs,s\right>.$$
        On the other hand, since $\chi=c(\nu_N\otimes\nu_f)$, we can compute $\mathcal{C}$ directly: In the ungraded case, 
        $$\mathcal{C}=-c(\nu_f)c(V)-c(V)c(\nu_f)=2g_0(\nu_f,V)=\epsilon\,\drm x^1(\nu_M)\circ \partial f.$$
        In the graded case, 
        $$\mathcal{C}=c(\nu_f)c(V)+c(V)c(\nu_f)=2g_0(\nu_f,V)=\epsilon\,\drm x^1(\nu_M)\circ \partial f.$$
        Hence, 
        \begin{equation}\label{eq:parallel-d}
            \left<c(\nu_N\otimes V)s,s\right>=\frac{1}{2}(\drm x^1(\nu_M)\circ \partial f)|s|^2.\tag{d}
        \end{equation}

        Combining \eqref{eq:parallel-a} -- \eqref{eq:parallel-d}, we obtain an integrated Lichnerowicz formula
        $$\int_N |\widehat Ds|^2-|\widehat \nabla s|^2-\frac{1}{4}\left(R_N+n(n-1)\right)|s|^2 = \frac{1}{2}\int_{\partial N} \epsilon\left<\mathcal Bs,s\right>+(H_{\partial N}-(n-1)\drm x^1(\nu_M)\circ \partial f)|s|^2.$$

        Since $\hat{D}s=0$, by Lemma \ref{lem:mean}, Proposition \ref{prp:B-estimate} and Lemma \ref{lem:norm-estimate}, we have 
        \begin{align*}
            0&\ge \int_N |\widehat Ds|^2-|\widehat \nabla s|^2-\frac{1}{4}\left(R_N+n(n-1)\right)|s|^2\\
            &=\frac{1}{2}\int_{\partial N} \epsilon\left<\mathcal Bs,s\right>+(H_{\partial N}-(n-1)\drm x^1(\nu_M)\circ \partial f)|s|^2\\
            &\ge\frac{1}{2}\int_{\partial N} (-\|\drm\hat\nu_f\|_1+H_{\partial M,g_1}\circ\partial f- (n-1)\drm x^1(\nu_M)\circ \partial f)|s|^2\\
            &\ge\frac{1}{2}\int_{\partial N} (-x^1H_{\partial M,g_0}\circ \partial f+H_{\partial M,g_1}\circ\partial f- (n-1)\drm x^1(\nu_M)\circ \partial f)|s|^2=0.
        \end{align*}
        Therefore, every inequality above is an equality. In particular, $\widehat \nabla s=0$, which implies that $|s|\ne 0$ everywhere. Hence, $R_N=-n(n-1)$, $H_{\partial N}=H_{\partial M}\circ \partial f$, and $\|\drm\hat\nu_f\|_1= x^1H_{\partial M,g_0}\circ \partial f$. By Lemma \ref{lem:norm-estimate}, $\partial f$ is a local isometry. Since $\partial M$ is a convex hypersurface in $(\mathbb{R}^n_+,g_0)$, it is simply connected, and thus $\partial f$ is an isometry. 
    \end{proof}

    \subsection{Parallel spinors lead to Killing spinors}

    Let $\Delta_n$ be the space of complex spinors on $\mathbb{R}^n$, which is a complex vector space of dimension $K=2^{\lfloor\frac{n}{2}\rfloor}$. An element $\pi\in\Delta_n$ can be viewed as a constant spinor on $M$, and thus we can define
    $$\tilde\pi:C^\infty(N,S)\to C^\infty(N,\slashed{S}_N), \quad \alpha\otimes\beta\mapsto\alpha\left<\beta,\pi\right>.$$
    Note that we are using the convention that the inner product is linear in the first spot. We use the same definition for both cases. 
    
    Recall that $\varphi\in C^\infty(N,\slashed{S}_N)$ is called a \emph{Killing spinor} if there exists a constant $\lambda\in\mathbb{C}$ such that 
    \begin{equation}\label{eq:Killing}
        \nabla^{\slashed{S}_N}_X\varphi=\lambda c(X)\varphi, \quad \forall X\in C^\infty(N,TN).
    \end{equation}
    The number $\lambda$ is called the \emph{Killing constant} of $\varphi$. 

    \begin{prp}\label{prp:nonzero}
        Suppose that there is a nontrivial solution to the boundary value problem
        $$\begin{cases}
            \widehat\nabla s=0 & \text{ on }N,\\
            \chi s=\epsilon s & \text{ on }\partial N.
        \end{cases}$$
        Then for any nonzero $\pi\in\Delta_n$, $\tilde{\pi}(s)$ is a nonzero spinor on $N$.
    \end{prp}

    \begin{proof}
        Since $\frac{\partial}{\partial x^1}$ is a unit vector, we have an orthogonal decomposition 
        $$\Delta_n=\Delta^+\oplus\Delta^-,$$
        where 
        \begin{equation}\label{eq:nonzero-spinor}
            \Delta^{\pm}=\left\{\pi\in\Delta_m:c\left(\frac{\partial}{\partial x^1}\right)\pi=\pm\sqrt{-1}\pi\right\}.
        \end{equation}
        For $\pi\in\Delta_m$, we can decompose $\pi=\pi^++\pi^-$ uniquely, where $\pi^{\pm}\in\Delta^{\pm}$. 
        
        In the ungraded case, suppose that 
        \begin{equation}\label{eq:nonzero-a}
            s=\sum_{i=1}^{K/2}\alpha_i\otimes\gamma_i^++\beta_i\otimes\gamma_i^-,\quad \gamma_i^{\pm}\in\Delta^{\pm}\text{ are constant spinors}.\tag{a}
        \end{equation}
        For simplicity, we write $s=\alpha\otimes\gamma^++\beta\otimes\gamma^-$ as a shorthand of \eqref{eq:nonzero-a}. Then for any $X\in TN$, we have 
        $$\nabla_Xs+c(X\otimes V)s=0,$$
        which implies that 
        \begin{equation}\label{eq:nonzero-b}
            \nabla_X^{\slashed S_N}\alpha\otimes\gamma^++\nabla_X^{\slashed S_N}\beta\otimes\gamma^-=\frac{\epsilon}{2\sqrt{-1}}(c(X)\alpha\otimes \gamma^+-c(X)\beta\otimes \gamma^-).\tag{b}
        \end{equation}

        In the graded case, if $s=\alpha\otimes\gamma^++\beta\otimes\gamma^-$, then 
        \[c\left(X\otimes\frac{\partial}{\partial x^1}\right)s=(-1)^{\deg\alpha}c(X)\alpha\otimes \gamma^+-(-1)^{\deg\beta} c(X)\beta\otimes \gamma^-.\]
        Note that because we are using a different metric on $TM$, there is no $\sqrt{-1}$ in this formula. Now we have 
        \begin{equation}\label{eq:nonzero-b'}
            \nabla_X^{\slashed S_N}\alpha\otimes\gamma^++\nabla_X^{\slashed S_N}\beta\otimes\gamma^-=-\frac{\epsilon}{2}((-1)^{\deg\alpha}c(X)\alpha\otimes \gamma^+-(-1)^{\deg\beta}c(X)\beta\otimes \gamma^-).\tag{b'}
        \end{equation}
        For simplicity, we write
        $$\epsilon'=\begin{cases}
            \frac{\epsilon}{2\sqrt{-1}}, & \text{ungraded}; \\ 
            -\frac{\epsilon}{2}, & \text{graded}.
        \end{cases}$$
        
        Motivated by \cite{chai2023scalar}, we consider the following subspaces of $\Delta_m$: 
        \begin{align*}
            \Theta&=\{\pi\in\Delta_m:\tilde\pi(s)=0\},\\
            \Theta_1&=\{\pi\in\Delta_m:\widetilde{\pi^+}(s)=\widetilde{\pi^-}(s)=0\},\\
            \Theta_2&=\{\pi\in\Delta_m:\widetilde{\pi^+}(s)=\widetilde{\pi^-}(s)=0 \text{ at some }y\in N\}.
        \end{align*}
        We want to show that 
        \begin{clm}
            $\Theta=\Theta_1=\Theta_2$. 
        \end{clm}

        First, we show that $\Theta_2\subseteq\Theta_1$. Suppose $\pi=\pi^++\pi^-\in\Theta_2$. Applying $\widetilde{\pi^{\pm}}$ to \eqref{eq:nonzero-b} or \eqref{eq:nonzero-b'}, we have for the ungraded case 
        $$\nabla_X^{\slashed S_N}\alpha\left<\gamma^+,\pi^+\right>=\epsilon'c(X)\alpha\left<\gamma^+,\pi^+\right>,\quad \nabla_X^{\slashed S_N}\beta\left<\gamma^-,\pi^-\right>=-\epsilon'c(X)\beta\left<\gamma^-,\pi^-\right>;$$
        for the graded case
        $$\nabla_X^{\slashed S_N}\alpha\left<\gamma^+,\pi^+\right>=(-1)^{\deg\alpha}\epsilon'c(X)\alpha\left<\gamma^+,\pi^+\right>,\quad \nabla_X^{\slashed S_N}\beta\left<\gamma^-,\pi^-\right>=-(-1)^{\deg\beta}\epsilon'c(X)\beta\left<\gamma^-,\pi^-\right>.$$
        
        In the ungraded case, we have 
        \begin{equation}\label{eq:nonzero-c}
            \nabla_X^{\slashed S_N}\widetilde{\pi^+}(s)=\epsilon'c(X)\widetilde{\pi^+}(s), \quad \nabla_X^{\slashed S_N}\widetilde{\pi^-}(s)=-\epsilon'c(X)\widetilde{\pi^-}(s).\tag{c}
        \end{equation}
        Since $\pi\in\Theta_2$, we know that $\widetilde{\pi^{\pm}}(s)=0$ at some $y\in N$. For any $y'\in N$, take a path $\gamma:[0,1]\to N$ connecting $y$ and $y'$. Then \eqref{eq:nonzero-c} becomes an ODE of $\widetilde{\pi^{\pm}}(s)\circ\gamma$. By uniqueness of the solution, we know that $\widetilde{\pi^{\pm}}(s)(y')=0$. Hence, $\pi\in\Theta_1$. 

        In the graded case, we decompose
        $$\widetilde{\pi^+}(s)=\widetilde{\pi^+}(s)^0+\widetilde{\pi^+}(s)^1,$$ 
        according to parity, and then 
        \begin{equation}\label{eq:nonzero-c'}
            \nabla_X^{\slashed{S}_N}\widetilde{\pi^+}(s)^0=-\epsilon'c(X)\widetilde{\pi^+}(s)^1, \quad \nabla_X^{\slashed{S}_N}\widetilde{\pi^+}(s)^1=\epsilon'c(X)\widetilde{\pi^+}(s)^0.\tag{c'}
        \end{equation}
        Since $\pi\in\Theta_2$, we know that $\widetilde{\pi^{+}}(s)^0=\widetilde{\pi^{+}}(s)^1=0$ at some $y\in N$. For any $y'\in N$, take a path $\gamma:[0,1]\to N$ connecting $y$ and $y'$. Then we have a system of ODEs of $\widetilde{\pi^{+}}(s)\circ\gamma$. By uniqueness of the solution, we know that $\widetilde{\pi^+}(s)(y')=0$. Similarly, $\widetilde{\pi^-}(s)(y')=0$ can be derived from
        \begin{equation}\label{eq:nonzero-c''}
            \nabla_X^{\slashed{S}_N}\widetilde{\pi^-}(s)^0=\epsilon'c(X)\widetilde{\pi^-}(s)^1, \quad \nabla_X^{\slashed{S}_N}\widetilde{\pi^-}(s)^1=-\epsilon'c(X)\widetilde{\pi^-}(s)^0.\tag{c''}
        \end{equation}
        Hence, $\pi\in\Theta_1$. 

        Next, we show that $\Theta\subseteq\Theta_2$. Suppose $\pi=\pi^++\pi^-\in\Theta$. Since $M$ is compact and $\partial f$ is surjective, we can find $y\in\partial N$ such that $\nu_f(y)=\frac{\partial}{\partial x^1}$. Then at $y$, by \eqref{eq:nonzero-a}, we have 
        $$\chi s=\bar c(\nu_N)(\alpha\otimes c(\nu_f)\gamma^++\beta\otimes c(\nu_f)\gamma^-)=\sqrt{-1}\bar c(\nu_N)(\alpha\otimes\gamma^+-\beta\otimes\gamma^-)$$
        if ungraded; 
        $$\chi s=\bar c(\nu_N)((-1)^{\deg\alpha}\alpha\otimes\gamma^+-(-1)^{\deg\beta}\beta\otimes\gamma^-)$$
        if graded. Since $\pi\in\Theta$, and since $\chi s=\epsilon s$, we have, at the point $y$, 
        $$\alpha\left<\gamma^+,\pi^+\right>-\beta\left<\gamma^-,\pi^-\right>=0$$
        if ungraded; 
        $$(-1)^{\deg\alpha}\alpha\left<\gamma^+,\pi^+\right>-(-1)^{\deg\beta}\beta\left<\gamma^-,\pi^-\right>=0$$
        if graded. In the latter case, we write down the even part and the odd part separately, and it follows that $\alpha\left<\gamma^+,\pi^+\right>-\beta\left<\gamma^-,\pi^-\right>=0$ still holds. 

        On the other hand, $\tilde{\pi}(s)=0$ implies that 
        $$\alpha\left<\gamma^+,\pi^+\right>+\beta\left<\gamma^-,\pi^-\right>=0.$$
        Hence, if ungraded, 
        $$\alpha\left<\gamma^+,\pi^+\right>=\beta\left<\gamma^-,\pi^-\right>=0$$
        which is the same as $\widetilde{\pi^{\pm}}(s)=0$, i.e., $\pi\in\Theta_2$. 

        Since it is clear that $\Theta_1\subseteq\Theta$, we conclude that $\Theta=\Theta_1=\Theta_2$. Now we want to show that 

        \begin{clm}
            $\Theta$ is invariant under $c(v)$ for all $v\in\mathbb{R}^n$. 
        \end{clm}

        Let $\pi=\pi^++\pi^-\in\Theta$ and let $v\in\mathbb{R}^n$ be a unit vector. Pick $y\in\partial N$ such that $\nu_f(y)=v$. Decompose $s$ as in \eqref{eq:nonzero-a}. Since $\pi\in\Theta_1$ and $\chi s=\epsilon s$, we have
        \begin{align*}
            0&=\widetilde{\pi^{+}}(\chi s)=\bar c(\nu_N)\left(\alpha\left<c(v)\gamma^+,\pi^+\right>+\beta\left<c(v)\gamma^-,\pi^+\right>\right)\\
            &=-\bar c(\nu_N)\left(\alpha\left<\gamma^+,c(v)\pi^+\right>+\beta\left<\gamma^-,c(v)\pi^+\right>\right)
        \end{align*}
        if ungraded;
        \begin{align*}
            0&=\widetilde{\pi^{+}}(\chi s)=\bar c(\nu_N)\left((-1)^{\deg\alpha}\alpha\left<c(v)\gamma^+,\pi^+\right>+(-1)^{\deg\beta}\beta\left<c(v)\gamma^-,\pi^+\right>\right)\\
            &=\bar c(\nu_N)\left((-1)^{\deg\alpha}\alpha\left<\gamma^+,c(v)\pi^+\right>+(-1)^{\deg\beta}\beta\left<\gamma^-,c(v)\pi^+\right>\right)
        \end{align*}
        if graded. In both cases, we have 
        $$(-1)^{\deg\alpha}\alpha\left<\gamma^+,c(v)\pi^+\right>+(-1)^{\deg\beta}\beta\left<\gamma^-,c(v)\pi^+\right>=0.$$
        Note that for $v=\frac{\partial}{\partial x^1}$, $(c(v)\pi)^{\pm}=c(v)\pi^{\pm}$; for $v\perp\frac{\partial}{\partial x^1}$, $(c(v)\pi)^{\pm}=c(v)\pi^{\mp}$. Hence, in both cases, 
        $$\widetilde{(c(v)\pi)^{\pm}}(s)=0.$$
        Therefore, $c(v)\pi\in\Theta_2$. This shows that $\Theta$ is invariant under $c(v)$ for all $v\in\mathbb{R}^n$. 

        Finally, we prove the proposition by showing that 
        \begin{clm}
            $\Theta=0$. 
        \end{clm}
        
        We know that $\Theta$ is an invariant subspace of the spinor representation $c:\Cl(\mathbb{R}^m)\to\End(\Delta_n)$, which is irreducible. Thus, $\Theta=0$ or $\Cl(\mathbb{R}^n)$. Using \eqref{eq:nonzero-a}, if we choose $\{\beta_i^{\pm}:1\le i\le K/2\}$ to be linearly independent, then there is at least one $\beta_i^{\pm}$ is not in $\Theta$ since $s\ne 0$. Hence, $\Theta=0$.
    \end{proof}

    \begin{cor}\label{cor:Killing}
        Let $s$ be a nontrivial solution as in Proposition \ref{prp:nonzero}. In the ungraded case, if $\pi\in\Delta^{\pm}$ is nonzero, where $\Delta^{\pm}$ is defined in \eqref{eq:nonzero-spinor}, then $\tilde\pi(s)$ is a nonzero Killing spinor with Killing constant $\pm\frac{\epsilon}{2\sqrt{-1}}$. In the graded case, if $\pi\in\Delta^{\pm}$ is nonzero, then $\tilde\pi(s)^0+\sqrt{-1}\tilde\pi(s)^1$ is a nonzero Killing spinor with Killing constant $\pm\frac{\epsilon}{2\sqrt{-1}}$.
    \end{cor}
    \begin{proof}
        This follows from \eqref{eq:nonzero-c}, \eqref{eq:nonzero-c'} and \eqref{eq:nonzero-c''} in the proof above. 
    \end{proof}
    \begin{cor}\label{cor:independent}
        Let $s$ be a nontrivial solution as in Proposition \ref{prp:nonzero}. If $\pi_1,\ldots,\pi_k$ are linearly independent in $\Delta_m$, then $\tilde\pi_1(s),\ldots,\tilde\pi_k(s)$ is linearly independent everywhere on $N$. 
    \end{cor}
    \begin{proof}
        Suppose that $c_1,\ldots,c_k\in\mathbb{C}$, and that $\sum_{j=1}^k c_j\tilde\pi_j(s)=0$ holds at some $y\in N$. Set
        $$\pi=\sum_{j=1}^k\bar c_j\pi_j.$$
        We have, by decomposing as in \eqref{eq:nonzero-a} in the proof above, 
        $$\tilde\pi(s)=\alpha\left<\gamma^+,\pi\right>+\beta\left<\gamma^-,\pi\right>=\sum_{j=1}^k c_j\left(\alpha\left<\gamma^+,\pi_j\right>+\beta\left<\gamma^-,\pi_j\right>\right)=\sum_{j=1}^kc_j\tilde\pi_j(s),$$
        which vanishes at $y$. Hence, $\pi\in\Theta_2=0$ by Proposition \ref{prp:nonzero}.
    \end{proof}

    \subsection{Existence of imaginary Killing spinors implies negative curvature}

    The following result is well known. We include a proof for completeness. 

    \begin{prp}\label{prp:negative}
        If there is a nontrivial Killing spinor on $N$ with Killing constant $\lambda=\pm\frac{\sqrt{-1}}{2}$, then $N$ has constant Ricci curvature $-(n-1)$. If there is a set of linearly independent Killing spinors $\{\varphi_i:1\le i\le 2^{\lfloor\frac{n}{2}\rfloor}\}$ on $N$ with Killing constant $\lambda_i=\pm\frac{\sqrt{-1}}{2}$, then $N$ has constant sectional curvature $-1$. 
    \end{prp}

    \begin{proof}
        Suppose that $\varphi\in C^\infty(N,\slashed{S}_N)$ satisfies \eqref{eq:Killing}. Let $\{\bar e_\alpha:1\le\alpha\le n\}$ be a local orthonormal frame of $TN$ defined on an open set $U\subseteq N$, with $\nabla^{TN} \bar e_\alpha=0$ at some $y\in U$. Then at $y$, we have $[\bar e_\alpha, \bar e_\beta]=0$ for any $\alpha,\beta$, and 
        \begin{align*}
            R^{\slashed S_N}(\bar e_\alpha,\bar e_\beta)\varphi&=\nabla^{\slashed{S}_N}_{\bar e_\alpha}\nabla^{\slashed{S}_N}_{\bar e_\beta}\varphi-\nabla^{\slashed{S}_N}_{\bar e_\beta}\nabla^{\slashed{S}_N}_{\bar e_\alpha}\varphi=\nabla^{\slashed{S}_N}_{\bar e_\alpha}(\lambda c(\bar e_\beta)\varphi)-\nabla^{\slashed{S}_N}_{\bar e_\beta}(\lambda c(\bar e_\beta)\varphi)\\
            &=\lambda c(\bar e_\beta)\nabla^{\slashed{S}_N}_{\bar e_\alpha}\varphi-\lambda c(\bar e_\beta)\nabla^{\slashed{S}_N}_{\bar e_\beta}\varphi=\lambda^2c(\bar e_\beta)c(\bar e_\alpha)\varphi-\lambda^2c(\bar e_\alpha)c(\bar e_\beta)\varphi\\
            &=\frac{1}{4}[c(\bar e_\alpha),c(\bar e_\beta)]\varphi.
        \end{align*}
        This holds for any orthonormal $\{\bar e_\alpha\}$. On the other hand, we know that 
        $$R^{\slashed S_N}(\bar e_\alpha,\bar e_\beta)\varphi=\frac{1}{2}\sum_{\gamma<\delta}\left<R^{TN}(\bar e_\alpha,\bar e_\beta)\bar e_\gamma,\bar e_\delta\right>c(\bar e_\gamma)c(\bar e_\delta)\varphi.$$
        Define
        \[\sigma_{\alpha,\beta}=[\bar e_\alpha,\bar e_\beta]-\sum_{\gamma,\delta}\left<R^{TN}(\bar e_\alpha,\bar e_\beta)\bar e_\gamma,\bar e_\delta\right>\bar e_\gamma\bar e_\delta\in C^\infty(U,\Cl_{N}).\]
        Here, $[\cdot,\cdot]$ is the commutator in the Clifford algebra. Then for any $\alpha,\beta$, we have $c(\sigma_{\alpha,\beta})\varphi=0$ on $U$. We compute the following element in $C^\infty(U,\Cl_{N})$:
        $$\sum_\beta\bar e_\beta\sigma_{\alpha,\beta}=2(n-1)\bar e_\alpha-\sum_{\beta,\gamma,\delta}\left<R^{TN}(\bar e_\alpha,\bar e_\beta)\bar e_\gamma,\bar e_\delta\right>\bar e_\beta\bar e_\gamma\bar e_\delta=(n-1)\bar e_\alpha+\frac{1}{2}\sum_{\beta,\gamma,\delta}R_{\alpha\beta\gamma\delta}\bar e_\beta\bar e_\gamma\bar e_\delta.$$
        In the last summation, the terms with distinct $\beta,\gamma,\delta$ add up to $0$ by Bianchi identity. Hence, 
        $$\sum_{\beta,\gamma,\delta}R_{\alpha\beta\gamma\delta}\bar e_\beta\bar e_\gamma\bar e_\delta=\sum_{\beta=\gamma\ne\delta}R_{\alpha\beta\gamma\delta}\bar e_\beta\bar e_\gamma\bar e_\delta+\sum_{\beta=\delta\ne\gamma}R_{\alpha\beta\gamma\delta}\bar e_\beta\bar e_\gamma\bar e_\delta=2\sum_\gamma R_{\alpha\gamma}\bar e_\gamma.$$
        Therefore, we have 
        $$X=\sum_\beta\bar e_\beta\sigma_{\alpha,\beta}=2(n-1)\bar e_\alpha+2\sum_\gamma R_{\alpha\gamma}\bar e_\gamma\in C^\infty(U,TN).$$
        Since $c(X)\phi=0$, we have 
        $$0=c(X)^2\phi=-|X|^2\phi$$
        on $U$. Since $\phi\ne 0$ on an open dense subset, we conclude that $X=0$ on $U$. Thus, 
        $$R_{\alpha\gamma}=-(n-1)\delta_{\alpha\gamma},$$
        i.e., $N$ has constant Ricci curvature $-1$. 

        Now suppose that $c(\sigma_{\alpha,\beta})\varphi_i=0$ for any $1\le i\le 2^{\lfloor\frac{n}{2}\rfloor}$. Then $\sigma_{\alpha,\beta}$ is in the kernel of the spinor representation. When $n$ is even, the kernel is trivial, and thus $\sigma_{\alpha,\beta}=0$. When $n$ is odd, the kernel of the spinor representation is the negative part, i.e., 
        $$\omega_{\mathbb{C}}\cdot\sigma_{\alpha,\beta}=-\sigma_{\alpha,\beta},$$
        where $\omega_{\mathbb{C}}$ is the complex volume element. Since $\omega_{\mathbb{C}}$ is odd while $\sigma_{\alpha,\beta}$ is even, we also have $\sigma_{\alpha,\beta}=0$ as before. Assume that $\alpha<\beta$ and $\gamma<\delta$. Then $\sigma_{\alpha,\beta}=0$ implies that 
        $$R_{\alpha\beta\gamma\delta}=\begin{cases}
            -1, & \text{if } \alpha=\gamma \text{ and } \beta=\delta \\ 
            0, & \text{otherwise}.
        \end{cases}$$
        Thus, $N$ has constant sectional curvature $-1$. 
    \end{proof}

    \subsection{Conclusion}

    Combining all results in this section, we conclude the following. 

    \begin{thm}\label{thm:rigidity}
        Suppose that there is a nontrivial solution to the boundary value problem
        $$\begin{cases}
            \widehat Ds=0 & \text{ on }N,\\
            \chi s=\epsilon s & \text{ on }\partial N.
        \end{cases}$$
        Then $N$ has constant sectional curvature $-1$, $H_{\partial N}=H_{\partial M,g_1}\circ \partial f$, and $\partial f$ is an isometry. 
    \end{thm}

    \begin{proof}
        It follows from Proposition \ref{prp:parallel}, Corollary \ref{cor:Killing}, Corollary \ref{cor:independent} and the second part of Proposition \ref{prp:negative}. 
    \end{proof}

    \section{The index theory}

    Throughout this section, $g=g_0$. Let $\epsilon=\pm 1$ be fixed. Set $$V=\frac{\epsilon}{2}\frac{\partial}{\partial x^1}$$ which is viewed as a constant vector field, and define modified connection as in Proposition \ref{prp:Lich}. Assume that 
    \begin{enumerate}
        \item The hypersurface $\partial M$ is convex in $(\mathbb{R}^n_+,g_0)$, i.e., $M$ is parabolically convex. 
        \setcounter{enumi}{3}
        \item The map $f$ has nonzero degree. 
    \end{enumerate}
    as in the main theorem \ref{thm:main}. Note that in the calculation we do not impose any curvature conditions. We first look at the ungraded case, and assume that $n$ is odd. 

    \begin{prp}\label{prp:elliptic}
        Assume that $n$ is odd. The boundary problem
        $$\begin{cases}
            \widehat Ds=0 & \text{ on }N,\\
            \chi s=\epsilon s & \text{ on }\partial N.
        \end{cases}$$
        is elliptic, and thus has a well defined index $\ind_\epsilon$. Moreover, $\ind_{-1}=-\ind_1$. 
    \end{prp}

    \begin{proof}
        It is easy to check that $\chi=\bar c(\nu_N)c(\nu_f)$ is a unitary involution, with 
        $$\chi \bar c(\nu_N)=\bar c(\nu_N)\chi,$$
        $$\chi \bar c(\bar e_\alpha)=-\bar c(\bar e_\alpha)\chi, \quad 1\le\alpha\le n-1,$$
        where $\{\bar e_\alpha:1\le\alpha\le n-1\}$ is a local orthonormal frame of $T\partial N$. Then by \cite[Example 7.26]{bar2012boundary}, $\chi$ is a boundary chirality operator. Thus, the local boundary condition $\chi s=\epsilon s$ is elliptic, and thus has a well defined index. 
        
        Consider a family of elliptic boundary value problems defined by 
        $$D_t=D-tnc(V), \quad 0\le t\le 1.$$
        Then $D_0=D$ and $D_1=\widehat{D}$. By homotopy invariance we know that $\ind_\epsilon D=\ind_\epsilon \widehat D$. Now $D$ is self-adjoint and the boundary conditions for $\epsilon=\pm 1$ are adjoint to each other. We conclude that $\ind_{-1}=-\ind_1$. 
    \end{proof}

    \begin{prp}\label{prp:index}
        Using the notations from Proposition \ref{prp:elliptic}, if $n$ is odd, then either $\ind_1$ or $\ind_{-1}$ is positive. 
    \end{prp}

    \begin{proof}
        Deform the metrics $\bar g$ and $g_0$ so that 
        \begin{itemize}
            \item On the boundaries, $\bar g|_{\partial N}$ and $g_0|_{\partial M}$ are unchanged. 
            \item The normal vectors $\nu_N$ and $\nu_{M,g_0}$ are unchanged. 
            \item There is a collar neighborhood $U$ of $\partial N$, and a collar neighborhood $V$ of $\partial M$ such that 
            $$U\cong \partial N\times [0,\delta), \quad V\cong \partial M\times [0,\delta)$$
            isometrically. 
        \end{itemize}
        Then $\ind_\epsilon$ is unchanged. Now we have $\nabla^{TM}\nu_M=0$, and thus, by Lemma \ref{lem:B}, $\mathcal{B}=0$. Write 
        $$S|_{\partial N}=S^+\oplus S^-,$$
        where $S^{\pm}$ are the eigen-subbundles of $\chi$ corresponding to eigenvalues $\pm 1$. Note that 
        $$\chi s=\pm s\implies \chi D^{\partial N}s=-D^{\partial N}\chi s=\mp D^{\partial N}s.$$
        We see that 
        $$D^{\partial N}:C^\infty(\partial N,S^{\pm})\to C^\infty(\partial N,S^{\mp}).$$
        By \cite[Theorem B.1]{bar2024scalar}, if $\ind_+=\ind_-=0$, then 
        $$\ind(D^{\partial N}:C^\infty(\partial N,S^{\pm})\to C^\infty(\partial N,S^{\mp}))=0.$$

        Now we make use of the assumption that $n$ is odd. Recall that if $n$ is odd, then the spinor representation
        $$c:\Cl(\mathbb{R}^n)\to\End(\Delta_n)$$
        is the irreducible representation of $\Cl(\mathbb{R}^n)$ such that $c(\omega_\mathbb{C})=1$. There is an embedding
        $$\iota:\Cl(\mathbb{R}^{n-1})\to\Cl(\mathbb{R}^n), \quad e_i\mapsto e_ie_n.$$
        Then $c\circ\iota$ is the spinor representation of $\Cl(\mathbb{R}^{n-1})$. Therefore, at any $y\in\partial N$, we have
        $$\slashed S_{N,y}=\slashed S_{\partial N,y},$$
        $$\iota:\Cl_{\partial N,y}\to\Cl_{N,y}, \quad \bar e_\alpha\mapsto \bar e_\alpha\nu_N.$$
        Under this identification, we have 
        \begin{itemize}
            \item $D^{\partial N}=-D_{\partial N}-\frac{1}{2}H_{\partial N}$, 
            \item $\chi=-\chi_{\partial N}$,
        \end{itemize}
        where $D_{\partial N}$ is the (twisted) Dirac operator on $S_{\partial N}=\slashed S_{\partial N}\otimes(\partial f)^*\slashed S_{\partial M}$, and 
        $$\chi_{\partial N}=c(\omega^{\partial N}_{\mathbb{C}})\otimes c(\omega^{\partial M}_{\mathbb{C}}).$$

        Define 
        $$S_{\partial N}^+=\slashed S_{\partial N}^+\otimes(\partial f)^*\slashed S_{\partial M}^++\slashed S_{\partial N}^-\otimes(\partial f)^*\slashed S_{\partial M}^-,$$
        $$S_{\partial N}^-=\slashed S_{\partial N}^+\otimes(\partial f)^*\slashed S_{\partial M}^-+\slashed S_{\partial N}^-\otimes(\partial f)^*\slashed S_{\partial M}^+.$$
        Then $S^{\pm}_{\partial N}=S^{\mp}$. Applying the Atiyah-Singer index theorem to each component of the closed manifold $\partial N$, we have 
        $$\ind(D_{\partial N}:C^\infty(\partial N,S^+_{\partial N})\to C^\infty(\partial N,S^-_{\partial N}))=\deg f\sum_{\pi_0(\partial N)}\chi(\partial M)\ne 0,$$
        where $\partial M$ is diffeomorphic to an even dimensional sphere. Hence, 
        $$\ind(D^{\partial N}:C^\infty(\partial N,S^{-})\to C^\infty(\partial N,S^{+}))\ne 0,$$
        which is a contradiction. Thus, $\ind_1$ and $\ind_{-1}$ cannot be both zero. By Proposition \ref{prp:elliptic}, one of them is positive. 
    \end{proof}

    Now we can finish the proof of the main theorem \ref{thm:main} for the odd dimensional case. 

    \begin{proof}[Proof of Theorem \ref{thm:main} (odd dimension).]
        By Proposition \ref{prp:index}, we can find $\epsilon=1$ or $-1$ such that
        $$\begin{cases}
            \widehat Ds=0 & \text{ on }N,\\
            \chi s=\epsilon s & \text{ on }\partial N.
        \end{cases}$$
        has a nontrivial solution. Then we apply Theorem \ref{thm:rigidity} to obtain the rigidity. 
    \end{proof}

    For the even dimensional case, we use the graded version. As in Proposition \ref{prp:elliptic}, we have 

    \begin{prp}
        The boundary problem
        $$\begin{cases}
            \widehat Ds=0 & \text{ on }N,\\
            \chi s=\epsilon s & \text{ on }\partial N.
        \end{cases}$$
        is elliptic. 
    \end{prp}

    Note that $\widehat{D}^*=\widehat{D}$, and the boundary condition is also self-adjoint. Thus, $\ind_\epsilon=0$. Hence, we should consider the operator 
    $$\widehat{D}^0:C^\infty(N,S^0)\to C^\infty(N,S^1),$$
    where $S=S^0\oplus S^1$ is the even-odd grading. 

    \begin{prp}\label{prp:index'}
        The index of 
        $$\begin{cases}
            \widehat D^0s=0 & \text{ on }N,\\
            \chi s= -s & \text{ on }\partial N.
        \end{cases}$$
        is $\deg f$. 
    \end{prp}

    \begin{proof}
        Consider a family of elliptic boundary value problems defined by 
        $$D_t=D-tnc(V), \quad 0\le t\le 1.$$
        Then $D_0=D$ and $D_1=\widehat{D}$. By homotopy invariance we know that $\ind D^0=\ind \widehat D^0$. According to \cite[Theorem A.3]{cecchini2024rigidity} and the remark after, the index of $$\begin{cases}
            D^0s=0 & \text{ on }N,\\
            \chi s= -s & \text{ on }\partial N.
        \end{cases}$$
        is $\deg f$. 
    \end{proof}

    Therefore, we can complete the proof in the even dimensional case. 

    \begin{proof}[Proof of Theorem \ref{thm:main} (even dimension).]
        By Proposition \ref{prp:index'}, 
        $$\begin{cases}
            \widehat Ds=0 & \text{ on }N,\\
            \chi s= -s & \text{ on }\partial N.
        \end{cases}$$
        has a nontrivial solution. Then we apply Theorem \ref{thm:rigidity} to obtain the rigidity. 
    \end{proof}

    \bibliographystyle{plain}
    \bibliography{ref}
\end{document}